\documentclass{elsarticle}
\usepackage{amssymb,amsmath}
\usepackage{overpic}
\usepackage{a4wide}
\usepackage{subfigure}

\newtheorem{lemma}{Lemma}
\newtheorem{theorem}[lemma]{Theorem}
\newtheorem{corollary}[lemma]{Corollary}
\newtheorem{definition}[lemma]{Definition}
\newenvironment{proof}{\noindent{\em Proof:}}{$\Box$~\\}

\def\be{\begin{eqnarray}}
\def\ee{\end{eqnarray}}
\def\bes{\begin{eqnarray*}}
\def\ees{\end{eqnarray*}}

\def\R{\mathbb R}       

\def\av{\mbox{\bf a}} \def\bv{\mbox{\bf b}} \def\cv{\mbox{\bf c}}
  \def\fv{\mbox{\bf f}}
\def\gv{\mbox{\bf g}}  
 \def\kv{\mbox{\bf k}} 
\def\mv{\mbox{\bf m}}  
  
\def\sv{\mbox{\bf s}}  
\def\vv{\mbox{\bf v}}  
\def\yv{\mbox{\bf y}}

\def\det{\operatorname{det}}    
\def\dist{\operatorname{dist}}  

\def\g{\gamma}

\def\l{\lambda}

\def\eps{\varepsilon}
\def\t{\tau}
\def\r{\rho}
\def\i{\mbox{i}}
\def\ol{\overline}

\def\phi{\varphi}

\begin{document}
\begin{frontmatter}

\title{Conchoid surfaces of spheres}

\author[mp]{Martin Peternell}
\author[mp]{David Gruber}
\author[js]{Juana Sendra}

\address[mp]{Institute of Discrete Mathematics and Geometry,\\
Vienna University of Technology, Vienna, Austria}
\address[js]{Dpto. Matem\'atica Aplicada a I. T. Telecomunicaci\'on, Univ. Polit\'ecnica de Madrid, Spain}

\begin{abstract}
The conchoid of a surface $F$ with respect to given fixed point $O$ is roughly speaking
the surface obtained by increasing the radius function with respect to $O$
by a constant.
This paper studies {\it conchoid surfaces of spheres} and shows that
these surfaces admit rational parameterizations.
Explicit parameterizations of these surfaces are constructed using the
relations to pencils of quadrics in $\R^3$ and $\R^4$. Moreover we
point to remarkable geometric properties of these surfaces and their
construction.
\end{abstract}


\end{frontmatter}

\section{Introduction}\label{intro:sec}

The conchoid is a classical geometric construction and dates back
to the ancient Greeks. Given a planar curve $C$, a fixed point $O$ (focus point)
and a constant distance $d$, the conchoid $D$ of $C$ with respect to $O$ at
distance $d$ is the ({Zariski} closure of the) set of points $Q$ in the line $OP$
at distance $d$ of a moving point $P$ varying in the curve $C$,
\begin{equation}
D=\{ Q\in OP \mbox{ with } P\in C, \mbox{ and } \overline{QP}=d\}^{*},
\end{equation}
{where the asterisk denotes the Zariski closure.}
For a more formal definition of conchoids in terms of
diagrams of incidence we refer to \cite{sendra1, sendra2}.

The definition of the conchoid surface to a given surface $F$ in space
with respect to a given point $O$ and distance $d$ follows analogous lines.

We aim at studying real rational surfaces in 3-space whose conchoid surfaces are also
rational and real. A surface $F\subset\R^3$ will be represented by a polar representation
$\fv(u,v)=\r(u,v)\kv(u,v)$, where $\kv(u,v)$ is a parameterization of the
unit sphere $S^2$. Without loss of generality we assume $O=(0,0,0)$.
Consequently their conchoid surfaces $F_d$ for varying distance
$d$ admit the polar representation $\fv_d(u,v)=(\r(u,v)\pm d)\kv(u,v)$.

Since we want to determine classes of surfaces whose conchoid surfaces
for varying distances are rational, we focus at rational polar surface
representations. Then the 'base' surface $F$ and its conchoids $F_d$
correspond to the same rational parameterization $\kv(u,v)$ of the unit
sphere $S^2$.
The following definition excludes possibly occurring cases where $F$ and $F_d$
are rational, but their rational parameterizations $\fv$ and/or $\fv_d$
are not corresponding to a rational representation $\kv(u,v)$ of $S^2$.

\begin{definition}
A surface $F$ is called {\em rational conchoid surface} with respect to
the focus point $O=(0,0,0)$ if $F$ admits a {\em rational polar representation}
$\r(u,v)\kv(u,v)$, with a rational radius function $\r(u,v)$ denoting the
distance function from $O$ to $F$ and a rational parameterization $\kv(u,v)$
of $S^2$.
\end{definition}

\paragraph{Contribution:}

The main contribution of this article is the study of the conchoid surfaces
of spheres. We prove that a sphere $F$ in $\R^3$
admits a rational polar representation
$\fv(u,v)=\r(u,v)\kv(u,v)$ with a rational radius function $\r(u,v)$
and a particular rational parameterization $\kv(u,v)$ of the unit sphere $S^2$,
independently of the relative position of the sphere $F$ and the focus point $O$.
This implies that the conchoids $G$ of $F$ with respect to
{any focus in $\R^3$} admit rational parameterizations.

It is remarkable that an analogous result to this contribution for spheres
does not exist for circles and conics in $\R^2$.
The conchoid curves of conics $C$ are only rational
if either $O\in C$ or $O$ coincides with one of $C$'s focal points.

Two constructions to prove the main result are presented. The first one uses
the cone model being introduced in Section~\ref{conemodel:sec} and studies
a pencil of quadrics in $\R^4$. This construction
is explicit and leads to a surprisingly simple solution and a rational
polar representation of a sphere. The second approach investigates pencils
of quadrics in $\R^3$ containing a sphere and a cone of revolution whose
base locus is a rational quartic with rational distance from $O$.

\subsection{The cone model}
\label{conemodel:sec}

Let $F$ be a surface in $\R^3$ and let $G$ be its conchoid surface at
distance $d$ with respect to the origin $O=(0,0,0)$ as focal point.
The construction of the conchoid surfaces $G$ of the 'base' surface $F$
is performed as follows.
Consider Euclidean 4-space $\R^4$ with coordinate axis $x,y,z$ and $w$,
where $\R^3$ is embedded in $\R^4$ as the hyperplane $w=0$.
Consider the quadratic cone $D:x^2+y^2+z^2-w^2=0$ in $\R^4$.
Further, let $A$ be the cylinder through $F$, whose generating
lines are parallel to $w$. Note that $A$ as well as $D$ are three-dimensional
manifolds in $\R^4$.
The conchoid construction of the 'base' surface $F$
is based on the study of the intersection $\Phi=A\cap D$, which
is typically a two-dimensional surface in $\R^4$.

For a given parameterization $\fv(u,v)$ of $F$ in $\R^3$, the cylinder
$A$ through $F$ admits the representation $\av(u,v,s) = (f_1,f_2,f_3,0)+s(0,0,0,1)$.
Let $F$ be a rational surface and $\fv(u,v)$ be rational. If the intersection
$\Phi = A\cap D$ is a rational surface in $\R^4$, then it is obvious that
$F$ admits a rational polar representation.
Let $\phi(a,b)=(\phi_1, \ldots, \phi_4)(a,b)$ be a rational
representation of $\Phi$ in $\R^4$, then $(\phi_1,\phi_2,\phi_3)(a,b)$ is
obviously a rational polar representation of $F$.
Since $\phi_4^2=\phi_1^2+\phi_2^2+\phi_3^2$ holds,
$\kv=1/\phi_4(\phi_1,\phi_2,\phi_3)$ is a rational
parameterization of $S^2$ and $\r(a,b)=\phi_4(a,b)$ is a rational radius function
of $F$. We summarize the construction.

\begin{theorem}\label{cone_model:theo}
The rational conchoid surfaces $F\subset\R^3$ are in bijective
correspondence to the rational 2-surfaces in the quadratic cone
$D:x^2+y^2+z^2-w^2=0$ in $\R^4$.
\end{theorem}
\begin{proof}
We proved already that for a rational surface $\Phi\subset D$, its orthogonal projection
$(\phi_1,\phi_2,\phi_3)$ onto $\R^3$ is a rational conchoid surface
with rational radius function $\phi_4$.
Conversely, any rational conchoid surface
$F$ with respect to $O$ is defined by a rational polar parameterization
$\r(u,v)\kv(u,v)$, with $\kv=(k_1,k_2,k_3)\in\R(u,v)^3$ and $\|\kv \|=1$.
The corresponding surface $\Phi\subset D$ is represented by
$\phi(u,v)=\r(k_1, k_2, k_3, 1)(u,v)$.
\end{proof}

The quadratic cone $D$ possesses universal parameterizations and we may use
them to specify all possible rational parameterizations of rational
conchoid surfaces. The construction starts with rational universal
parameterizations of the unit sphere $S^2$. Following \cite{dietz} we
choose four arbitrary polynomials $a(u,v)$, $b(u,v)$, $c(u,v)$ and
$d(u,v)$ without common factor.
Let
$$
\alpha = 2(ac+bd), \beta = 2(bc-ad), \gamma = a^2+b^2-c^2-d^2, \delta = a^2+b^2+c^2+d^2,
$$
then $\kv(u,v) = \frac{1}{\delta}(\alpha,\beta,\gamma)$ is a rational
parameterization of the unit sphere $S^2$.
Thus $\phi(u,v)=\rho(u,v)(\alpha, \beta, \gamma, \delta)$ with a non-zero
rational function $\rho(u,v)$ is a rational parameterization of
a two-dimensional surface $\Phi\subset D$. Consequently
$$
\fv(u,v) = \rho(u,v)
\left(\frac{\alpha}{\delta}, \frac{\beta}{\delta}, \frac{\gamma}{\delta} \right)(u,v)
= \rho(u,v)\kv(u,v),
$$
is a rational polar representation of a rational conchoid surface $F$ in $\R^3$,
with $\rho(u,v)$ as radius function and $\kv(u,v)$ as rational parameterization
of the unit sphere $S^2$. It is sufficient to consider polynomials
and the construction reads as follows.

\begin{corollary}\label{univ_para_conch:theo}
Given six relatively prime polynomials $a(u,v)$, $b(u,v)$, $c(u,v)$, $d(u,v)$,
and $r(u,v)$ and $s(u,v)$, a universal parameterization of a rational
2-surface $\Phi\subset D$ in $\R^4$ is given by
\begin{equation}\label{univpararatconch}
\phi(u,v)=\frac{r}{s}\left( 2(ac+bd),2(bc-ad),a^2+b^2-c^2-d^2,
a^2+b^2+c^2+d^2 \right)(u,v).
\end{equation}
Consequently, a universal rational parameterization of a rational
conchoid surface reads
\begin{equation}
\fv(u,v)=\frac{r}{s(a^2+b^2+c^2+d^2)}
\left(2(ac+bd),2(bc-ad),a^2+b^2-c^2-d^2 \right)(u,v).
\end{equation}
\end{corollary}

This is a general result about all rational parameterizations
of rational conchoid surfaces. For a particular given
rational surface $F$ it is difficult
to decide whether the intersection $\Phi=D\cap W$ admits rational
parameterizations or not. Typically the surface $\Phi$ is not rational.
Nevertheless, there are interesting non-trivial
cases where $\Phi$ admits rational parameterizations.

In \cite{pet10} it has been proved that conchoids
of rational ruled surfaces $F$ are rational.
We give a hint how this result can be proved with help of the cone model $D$ and Theorem~\ref{cone_model:theo}.
If $F$ is a ruled surface, the cylinder $A\subset\R^4$ carries a one-parameter
family of planes parallel to the $w$-axis. These planes pass through the
generating lines of $F$. This implies that typically the intersection
$\Phi=A\cap D$ carries a one-parameter family of conics obtained as intersections
of the mentioned planes with $D$. This family of conics is rational, and
it is known (\cite{pet_thesis, schicho_thesis}) that there exist
rational parameterizations $\phi(u,v)$ of $\Phi$. Thus the conchoids
of real rational ruled surfaces are rational.

In this context we mention a trivial but useful statement which we prove for completeness.
\begin{lemma}\label{ratcrv_rotcone:theo}
Given a rational curve $C$ with parameterization $\cv(t)$ on a rotational cone $D$,
then the distance $\|\cv(t)-\vv\|$ between the curve $C$ and the vertex $\vv$ of $D$
is a rational function.
\end{lemma}

\begin{proof}
We use a special coordinate system with $\vv$ at the origin, and $z$
as rotational axis of $D$. This implies that $D$ is the zero set of
$x^2+y^2-\g^2z^2$. Without loss of generality we let $\g=1$.
The given curve $C$ admits therefore a rational parameterization
$\cv(t)=(c_1,c_2,c_3)(t)$ satisfying $c_1^2+c_2^2=c_3^2$.
Obviously one obtains $\|\cv(t)\| = \sqrt 2 c_3(t)$ being rational.
\end{proof}

\section{Conchoids of spheres}
\label{conch_sphere:sec}

Given a sphere $F$ in $\R^3$ and an arbitrary focus point $O$, the question arises
if there exists a rational representation $\fv(u,v)$ of $F$ with the property
that $\|\fv(u,v)\|$ is a rational function of the parameters $u$ and $v$.
To give a constructive answer to this question we describe an approach using the
cone-model presented in Section~\ref{conemodel:sec}. Later on in Section~\ref{revcone:sec} we study a different method  working in $\R^3$ directly.
There are several relations between these methods which will be discussed along
their derivation.

Let $F$ be the sphere with center $\mv=(m,0,0)$ and radius $r$, and let
$O=(0,0,0)$. Thus $F$ is given by
\begin{equation}\label{sphere:eq}
F: (x-m)^2+y^2+z^2-r^2=0.
\end{equation}
If $m=0$, the center of $F$ coincides with $O$. In this trivial
situation the conchoid surface of $F$ is reducible and consists of two spheres,
where one might degenerate to $F$'s center if $d=r$.
If $m^2-r^2=0$, the focal point $O$ is contained in $F$. To construct a rational
polar representation, we make the ansatz $\fv(u,v) = \rho(u,v)\kv(u,v)$ with
$\kv(u,v)=(k_1,k_2,k_3)(u,v)$ and $\|\kv(u,v)\|=1$ and
an unknown radius function $\rho(u,v)$.
Plugging this into~\eqref{sphere:eq}, we obtain a rational
polar representation with rational radius function $\rho(u,v)=2mk_1(u,v)$.
Note that in this case the conchoid is irreducible and rational.

\subsection{Pencil of quadrics in $\R^4$}
\label{quad_pencil2:sec}

Consider the Euclidean space $\R^4$ with coordinate axes $x,y,z$ and $w$ and
let $\R^3$ be embedded as the hyperplane $w=0$.
Let a sphere $F\subset\R^3$ be defined by~\eqref{sphere:eq} and $O=(0,0,0)$.
To study the general case we assume $m\not=0$ and $m^2\not=r^2$.
The equation of the cylinder $A\subset\R^4$ through $F$ with
$w$-parallel lines agrees with the equation of $F$ in $\R^3$,
\begin{equation}
A: (x-m)^2+y^2+z^2-r^2=0.
\end{equation}
Consider the pencil $Q(t) = A+tD$ of quadrics in $\R^4$, spanned by $A$ and
the quadratic cone $D:x^2+y^2+z^2=w^2$ from Section~\ref{conemodel:sec}.
Any point $\ol X=(x,y,z,w)\in D$ has the property that the distance from
$X=(x,y,z)$ to $O$ in $\R^3$ equals $w$.
We study the geometric properties of the del Pezzo surface $\Phi=A\cap D$
of degree four, the base locus of the pencil of quadrics $Q(t)$.
According to Theorem~\ref{cone_model:theo}, the sphere $F$
is a rational conchoid surface exactly if $\Phi$ admits rational
parameterizations.

Besides $A$ and $D$ there exist two further singular quadrics in $Q(t)$.
These quadrics are obtained
for the zeros $t_1 = -1$ and $t_2 = r^2/\g^2$ of the characteristic polynomial
\begin{equation}\label{charpoly_r4:eq}
\det(A+tD)=-(1+t)^2 t (\g^2 t - r^2), \mbox{ with } \g^2 = m^2-r^2 \not=0.
\end{equation}
The quadric corresponding to the twofold zero $t_1=-1$ is a cylinder
\begin{equation}\label{parcyl:eq}
R: w^2 - 2mx + m^2-r^2=0.
\end{equation}
Its directrix is a parabola in the $xw$-plane and its two-dimensional generators
are parallel to the $yz$-plane.
The singular quadric $S$ corresponding to $t_2=r^2/\g^2$ is a quadratic cone and reads
$$
%
S: \left( x-\frac{m^2-r^2}{m}\right)^2+y^2+z^2 = \frac{r^2}{m^2}w^2.
$$
Its vertex is the point $O'=(\frac{m^2-r^2}{m},0,0,0)$. The intersections
of $S$ with three-spaces $w=c$ are spheres $\sigma(c)$,
whose top view projections in $w=0$
are centered at $O'$ and their radii are $rc/m$.
The intersections of $D$ with three-spaces $w=c$ are spheres $d(c)$ whose top view projections in $w=0$ are centered at $O$ with radii $c$.
The intersections $k(c)=s(c)\cap d(c)$ of these spheres ($w=c$) are circles
in planes $x=(c^2+m^2-r^2)/(2m)$.
Thus $\Phi$ contains a family of conics, whose top
view projections are the circles $k(c)$. The conics in $\Phi$ are contained
in the planes
$$
\eps(c): x=\frac{c^2+m^2-r^2}{2m}, w=c.
$$

The half opening angle $\delta$ of $D$ with respect to the $w$-axis is $\pi/4$, thus
$\tan \delta=1$. The half opening angle $\sigma$ of $S$ is given by
$\tan\sigma=r/m$, see Figure~\ref{quadpencilr4:fig}. Applying the scaling
$$
(x', y', z', w')=(fx, fy, fz, w), \mbox{ with } f=\frac{r}{m}
$$
in $\R^4$ maps $D$ to a congruent copy of $S$.
Consider a point $\overline X=(x,y,z,w)$ in $\Phi=A\cap D$ and its projection
$X=(x,y,z)$ in $F$.
The distance $\dist(X,O)$ of $X$ to $O$ in $\R^3$ is $w$.
For the distance $\dist(X,O')$ between $X$ and $O'$ we consequently obtain
\begin{equation}\label{dist_frac:eq}
\dist(X,O')=\frac{r}{m}\dist(X,O), \mbox{ for all } X\in F.
\end{equation}

\paragraph{Remark on the circle of Apollonius}
Note that $O'$ is the inverse point of $O$ with respect to the sphere $F$.
It is an old result by Apollonius Pergaeus (262--190 b.c.)
that the set of points $X$ in the plane having constant ratio of distances
$f=d/d'$, with $d=\dist(O,X)$ and $d'=\dist(O',X)$, from two given fixed
points $O$ and $O'$, respectively, is a circle $k$, see Fig.~\ref{apoll_circle_fig}.
Rotating $k$ around the line $OO'$ gives the sphere $F$ and $O$ and $O'$ are inverse
points with respect to $F$ (and the circle $k$).

If we consider a varying constant ratio $f$, one obtains a family of spheres
$F(f)$ with inverse points $O$ and $O'$ which form an elliptic pencil of
spheres. Their centers are on the line $OO'$.
Ratio 1 ($d=d'$) corresponds to the bisector plane of $O$ and $O'$.

\begin{figure}
\subfigure[Sketch of the quadric pencil in $\R^4$ ]{\label{quadpencilr4:fig}
\begin{overpic}[width=.55\textwidth]{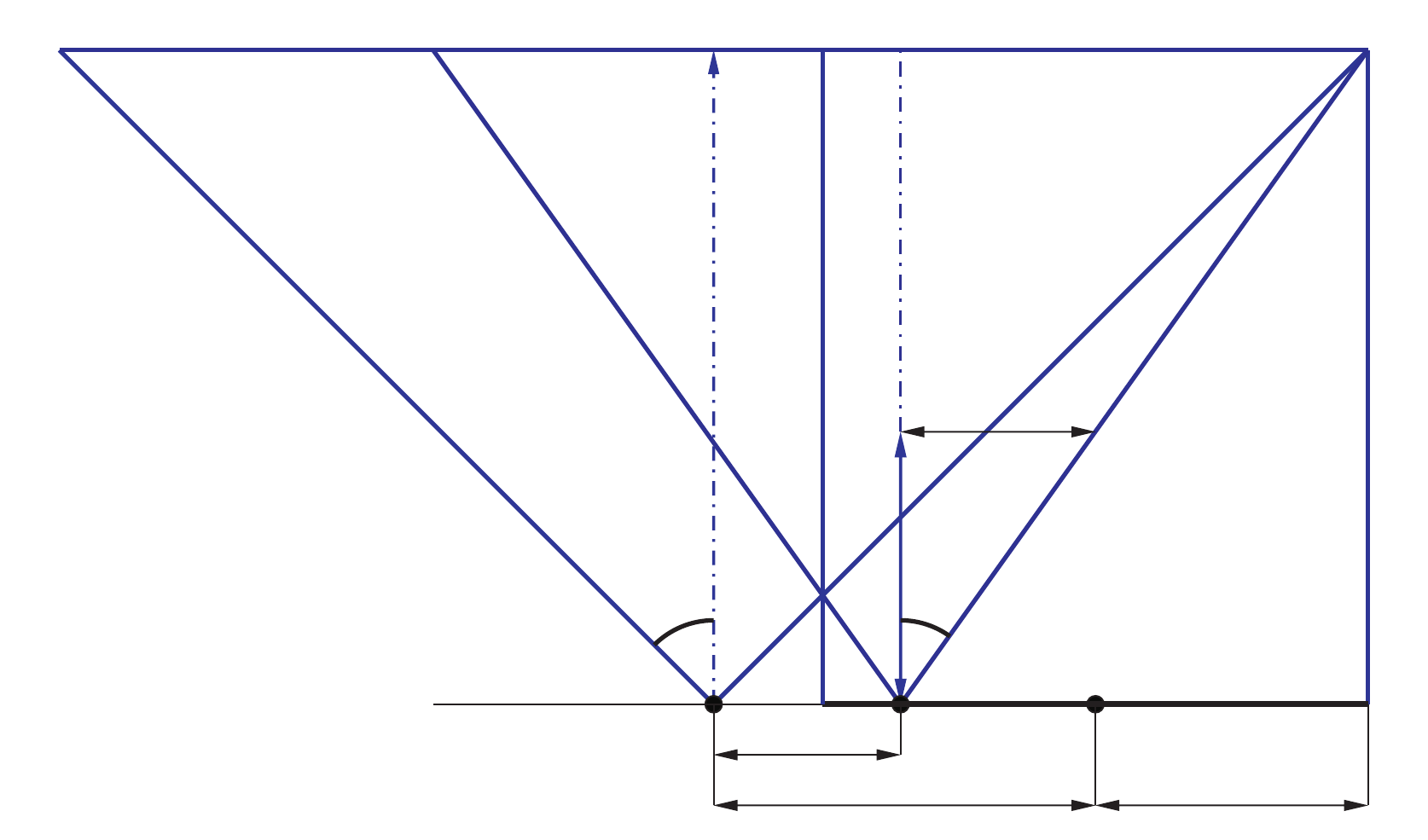}%
{\small
\put(25,10){$\R^3$} \put(52,52){$w$} \put(90,30){$A$}
\put(45,5){$O$} \put(65,5){$O'$} \put(75,11){$M$} \put(90,11){$F$}
\put(85,-1){$r$} \put(60,-1){$m$} \put(66,31){$\frac{r^2}{m}$}
\put(64,20){$r$} \put(53,5){$\frac{\g^2}{m}$}
\put(45.5,15.5){$\delta$} \put(64.5,15.5){$\sigma$}
\put(12,50){$D$} \put(37,50){$S$}
}
\end{overpic}
}
\hfil
\subfigure[Apollonius circle]{\label{apoll_circle_fig}
\begin{overpic}[width=.4\textwidth]{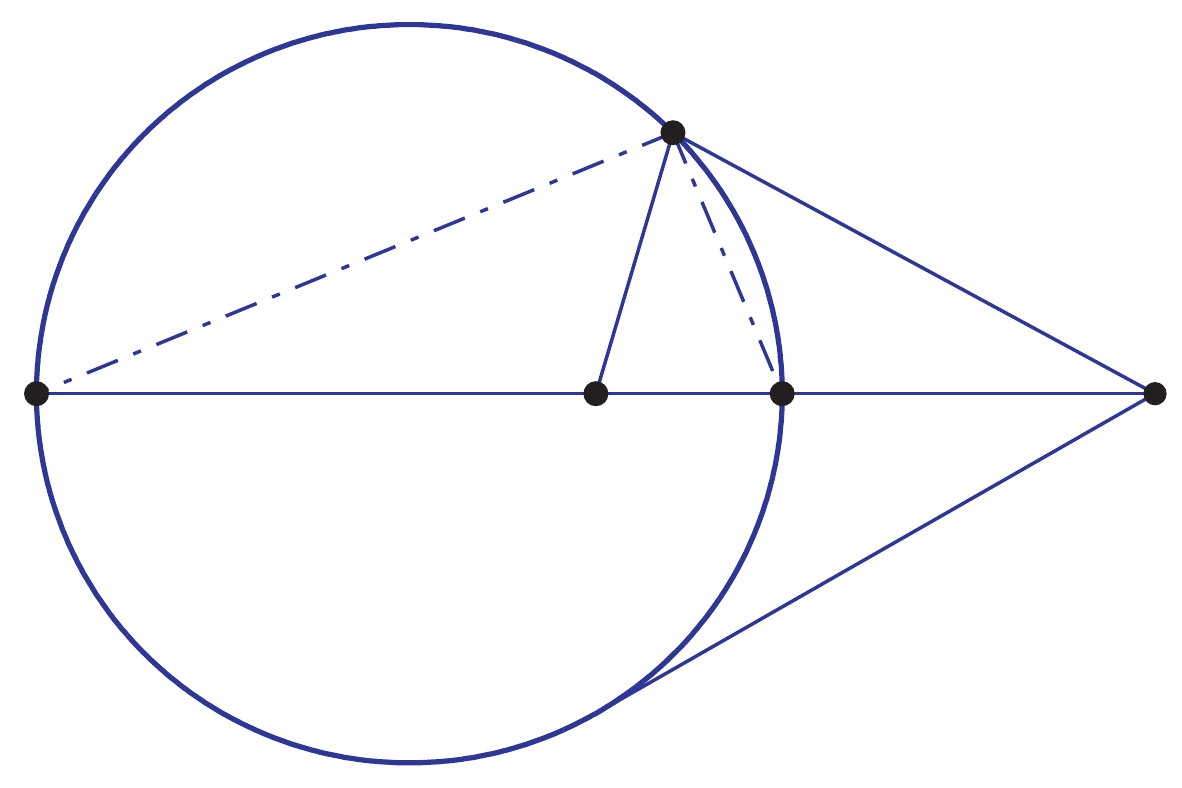}%
{\small
\put(98,28){$O$}
\put(51,28){$O'$} \put(57,57){$X$}
\put(10,58){$k$}
}
\end{overpic}
}
\caption{Pencil of quadrics in $\R^4$ and Apollonius circle}
\label{fig3}
\end{figure}
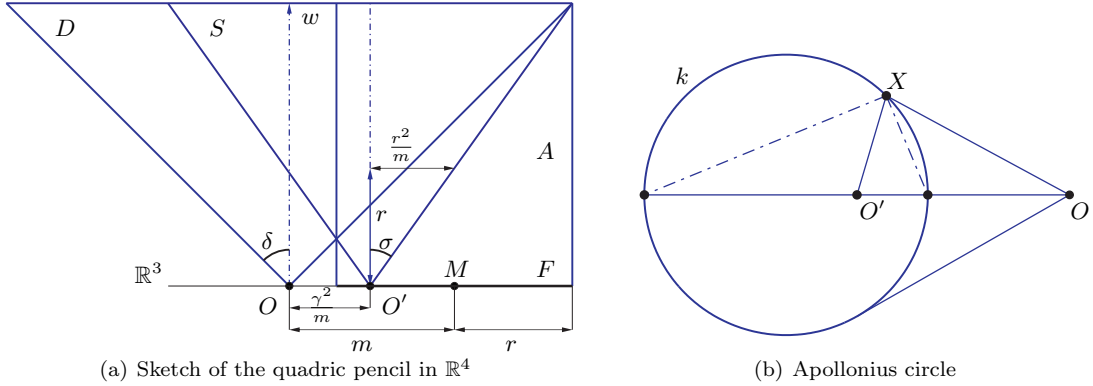

\subsection{A rational quartic on the sphere }
\label{ratquart:sec}

The pencil of quadrics $Q(t)$ in $\R^4$ spanned by the sphere $F$ and the
cone $D$ contains the cylinder $R$.
Expressing the variable $x$ from~\eqref{parcyl:eq} one gets
\begin{equation}\label{x-coord_c:eq}
x=\frac{w^2+m^2-r^2}{2m},
\end{equation}
and inserting this into $D$ results in the polynomial
\begin{equation}\label{conic_pencil:eq1}
\alpha(w): 4m^2(y^2+z^2) + p(w)=0, \mbox{ with }
p(w)= w^4 - 2w^2(m^2+r^2)+(m^2-r^2)^2.
\end{equation}
Considering $y$ and $z$ as variables, $\alpha(w)$ is
a one-parameter family of conics (circles) in the $yz$-plane,
depending {\it rationally} on the parameter $w$.
The circles $\alpha(w)$ do not possess real points for all $w$,
but there exist intervals determining families of real circles $\alpha(w)$.
To obtain real circles one has to perform a re-parameterization $w(u)$
within an appropriate interval.
The factorization of $p(w)$ reads
$$
p(w)=(w+a)(w-a)(w+b)(w-b), \mbox{ with } a=m+r, \mbox{ and } b=m-r.
$$
If $O$ is outside of $F$, thus $m>r$, the polynomial $-p(w)$ is positive
in the interval $[m-r, m+r]$. Thus a possible re-parameterization is
\begin{equation}\label{repara:eq}
w(u) = \frac{a u^2 + b}{1+u^2} = \frac{u^2(m+r)+m-r}{1+u^2}.
\end{equation}
Otherwise we could re-parameterize over another appropriate interval.
Additionally we note that if $O$ is inside of $F$, the inverse point $O'$
is outside of $F$. Since equation~\eqref{dist_frac:eq} holds for the distances
of a point $X\in F$ to $O$ and $O'$,
we can exchange roles and perform the computation for the point $O'$.

We return to the family of conics $\alpha(w)$.
Substituting \eqref{repara:eq} into \eqref{conic_pencil:eq1} leads to
a family of real conics
\begin{equation}\label{repara2:eq}
\alpha(u): y^2+z^2 = \frac{4r^2 u^2}{m^2(1+u^2)^4}(au^2+m)(mu^2+b).
\end{equation}
We are looking for rational functions $y(u)$ and $z(u)$ satisfying
\eqref{repara2:eq} identically.
Therefore we introduce
auxiliary variables $\tilde y$ and $\tilde z$ by the relations
$y = 2\tilde y ru/(m(1+u^2)^2)$ and $z = 2\tilde z ru/(m(1+u^2)^2)$.
We obtain $\tilde y^2 + \tilde z^2 =  (au^2+m)(mu^2+b)$.
Factorizing left and right hand side of this equation
results in a linear system to determine $\tilde y$ and $\tilde z$,
\begin{eqnarray*}
\tilde y + \i \tilde z &=& (\sqrt a u + \i \sqrt m)(\sqrt m u + \i \sqrt b), \\
\tilde y - \i \tilde z &=& (\sqrt a u - \i \sqrt m)(\sqrt m u - \i \sqrt b).
\end{eqnarray*}
The solution $\tilde y = \sqrt m (\sqrt a u^2 - \sqrt b)$,
$\tilde z = u(m + \sqrt{ab})$ finally leads to
\begin{equation}\label{conic_fam_par:eq}
y(u) = \frac{2r\sqrt m u}{m(1+u^2)^2}\left( \sqrt{a}u^2-\sqrt{b}\right),
\mbox{ and }
z(u) = \frac{2ru^2}{m(1+u^2)^2}\left( m+\sqrt{ab}\right),
\end{equation}
which is a rational parameterization of a curve in the $yz$-plane, following
the family of conics $\alpha(w)$.

We note that any real rational family of conics possesses
real rational parameterizations, see for instance \cite{pet_thesis, schicho_thesis}.
The solution \eqref{conic_fam_par:eq} together with \eqref{x-coord_c:eq} determines
a curve $C \subset F$ which possesses the rational distance function
\begin{equation}\label{ratdist1}
\| \cv(u) \| = w(u) = \frac{u^2(m+r)+ (m-r)}{1+u^2}
\end{equation}
with respect to $O$. Its parameterization is
\begin{equation}\label{quartic_on_s2:eq}
\cv(u) = \frac{1}{m(1+u^2)^2}
\left(
\begin{array}{c}
u^4m(m+r) + 2u^2(m^2-r^2)+m(m-r)\\
2r\sqrt{m}u(u^2\sqrt{m+r} - \sqrt{m-r})\\
2ru^2(m + \sqrt{m^2-r^2})\\
\end{array}
                    \right).
\end{equation}

\begin{theorem}\label{ratquartic:theo}
Let $F$ be a sphere and let $O$ be an arbitrary point in $\R^3$.
Then there exists a rational quartic curve $C\subset F$ and a rational parameterization
$\cv(u)$ of $C$ such that the distance of $C$ to $O$ is a rational function in
the curve parameter $u$.
\end{theorem}

Rotating $C$ around the $x$-axis leads to a rational polar representation
$r(u,v)\kv(u,v)$ of $F$ with rational distance function $\r(u,v)=w(u)$
from $O$. The quartic curve $C$ together
with this parameterization is illustrated in Fig.~\ref{basecrv_fig}.
Fig.~\ref{conch_fig} displays a sphere $F$ together with both conchoid surfaces
$G_1$ and $G_2$ for distances $d$ and $-d$ with respect to $O$.
We summarize the presented construction.
\begin{theorem}\label{ratconchoid:theo}
Spheres in $\R^3$ admit rational polar representations with respect to
any focus point $O$. This implies that the conchoid surfaces
of spheres admit rational parameterizations. The construction is based on rational
quartic curves on $F$ with rational distance from $O$.
\end{theorem}

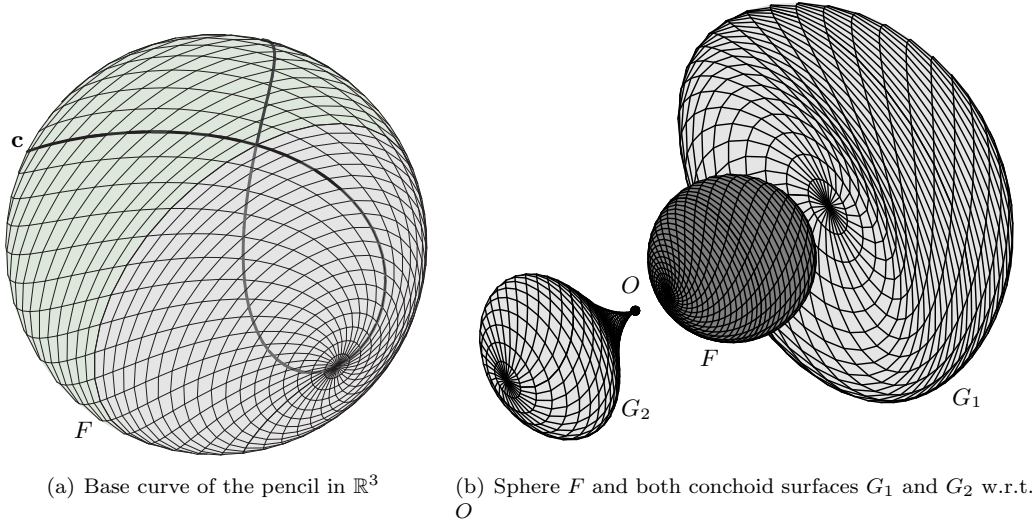
\begin{figure}
\subfigure[Base curve of the pencil in $\R^3$]{\label{basecrv_fig}
\begin{overpic}[width=.4\textwidth]{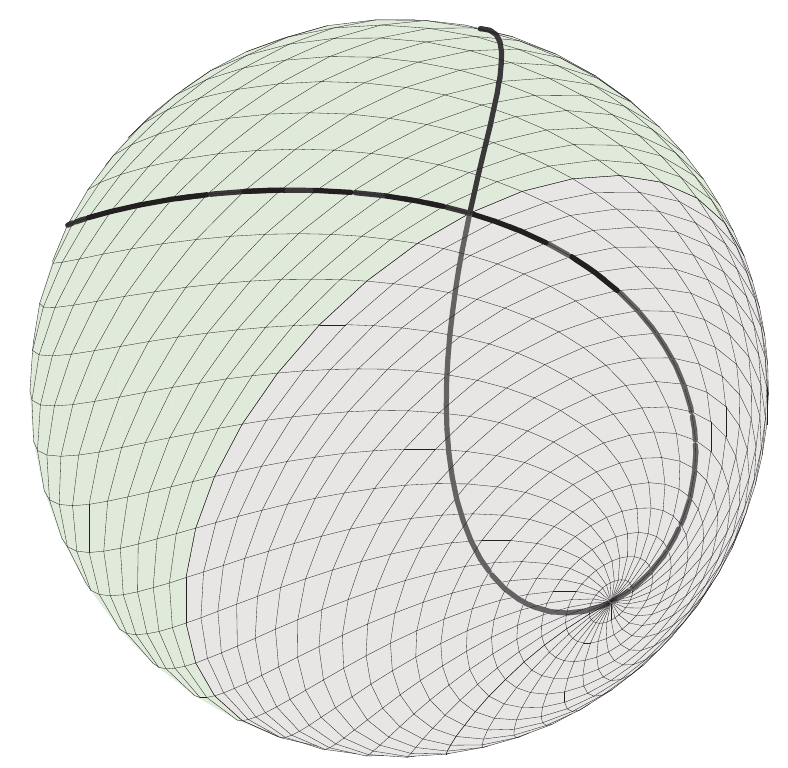}%
{\small \put(19,5){$F$} \put(5,70){$\cv$}}
\end{overpic}
}
\subfigure[Sphere $F$ and both conchoid surfaces $G_1$ and $G_2$ w.r.t. $O$]{\label{conch_fig}
\begin{overpic}[width=.5\textwidth]{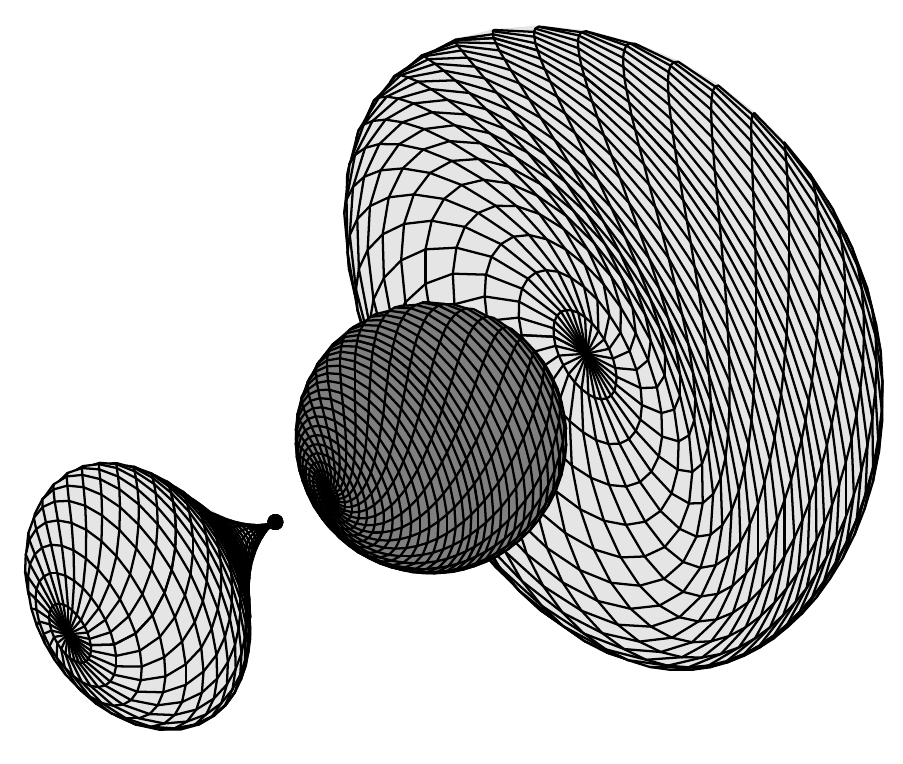}%
{\small \put(28,30){$O$} \put(42,17){$F$} \put(87,10){$G_1$} \put(28,8){$G_2$}}
\end{overpic}
}
\caption{Rational polar representation of a sphere and its conchoid surfaces}
\label{fig2}
\end{figure}

\paragraph{Rationality and Uni-Rationality}

The construction performed in Section~\ref{ratquart:sec} yields a rational
parameterization $\fv(u,v)$ of the sphere $F$ with rational radius function
$\rho(u,v)$, given by~\eqref{ratdist1}, such that $\fv(u,v)=\rho(u,v)\kv(u,v)$,
where $\kv(u,v)$ is an improper parameterization of the unit sphere $S^2$.
This means that typically a point $X\in F$ corresponds to two points $(u,v_1)$
and $(u,v_2)$ in the parameter domain. Rotating the curve $C$ around the $x$-axis,
the sphere $F$ is double covered.

The conchoid surface $G$ of $F$ at distance $d$ typically consists of two surfaces
$G_1$ and $G_2$, which admit the rational parameterizations
\begin{equation}\label{conchoid:parameterization}
\gv_1 = (\rho(u,v)+d)\kv(u,v), \mbox{ and } \gv_2 = (\rho(u,v)-d)\kv(u,v),
\end{equation}
for positive and negative distance. The conchoid $G=G_1\cup G_2$ is an
irreducible algebraic surface of order six.
It is {\it not bi-rational} equivalent to the
projective plane but each component $G_1$ as well as $G_2$ admits improper
{\it rational parameterizations}. These components $G_1$ and $G_2$ are called
{\it uni-rational}. This is not a contradiction to Castelnuovo's theorem since
we are not working over an algebraically closed field but over the field of real
numbers $\R$.

Let us consider an example to illustrate these properties. We consider the sphere
$F$ with center $\mv=(3/2,0,0)$ and radius $r=1$, and compute its conchoid $G$
for variable distance $d$. We obtain parameterizations $\gv_1(u,v)$ and $\gv_2(u,v)$ from
equation~\eqref{conchoid:parameterization} for the real uni-rational varieties
$G_1$ and $G_2$. The algebraic variety $G=G_1\cup G_2$ is given by the equation
\begin{equation}
\begin{array}{ll}
G: & (x^2+y^2+z^2)(4(x^2+y^2+z^2)-12x+5)^2 \\
&  + d^2(40(x^2+y^2+z^2)-144x^2+96x(x^2+y^2+z^2)-32(x^2+y^2+z^2)^2)\\
&  + 16d^4(x^2+y^2+z^2) =0.\\
\end{array}
\end{equation}


\paragraph{Remarks on the parameterization}
The rational quartic $C$ on $F$ is of course not unique but depends on the re-parameterization~\eqref{repara:eq}. An admissible rational re-parameterization
of a real interval is of even degree. Let us consider a quadratic re-parameterization.
Since $\alpha$ is of degree four in $w$, the re-parameterized family is typically
of degree $\leq 8$ in $u$.
This implies that the solutions $y(u)$ and $z(u)$ are of degree $\leq 4$, which holds
also for $x(u)$ because of \eqref{x-coord_c:eq}.
The coefficient functions $\cv(u) = (x,y,z)(u)$ determine a rational quartic
$C$ on $F$, with rational norm $\|\cv\|=w(u)$.

Different choices of the interval and a quadratic re-parameterization will typically
result in different quartic curves on $F$.
In \eqref{repara:eq} we have chosen the largest possible interval and a rational
function satisfying $w(-u)=w(u)$ and obtained
the curve $C$ through antipodal points of $F$.
By rotating we obtain the full sphere, doubly covered.

For any quadratic re-parameterization, the quartic $C$ is the base locus of
a pencil of quadrics $Q(t) = F+tK$, spanned by the sphere $F$ and, for instance,
the quadratic projection cone $K$ with vertex at $C$'s double point.

The particular choice~\eqref{repara:eq} implies that the quartic $C$ is symmetric
with respect to the $xz$-plane. This holds since $u$ appears only with even powers
in $x$ and $z$, thus we have $x(-u)=x(u)$ and $z(-u)=z(u)$.
The orthogonal projection of $C$ to the $xz$-plane is doubly covered, thus a conic.
In this case $(x,z)(u)$ parameterizes a parabola, because of the factor $(1+u^2)^2$ in
$\cv(u)$'s denominator. This implies that the pencil $Q(t)$ can also be spanned by
the sphere $F$ and the parabolic cylinder $P$ passing through $C$, whose generating
lines are parallel to $y$. It can be proved that all quadrics in $Q(t)$ except $P$
are rotational quadrics with parallel axes. This implies that $K$ is a rotational cone,
and the remaining singular quadric $L$ is a rotational cone, too.
For the particular choice~\eqref{repara:eq} and for the generalized construction performed in Section~\ref{revcone:sec}, the rotational cone $L$ has the vertex $O$.
We note that for any admissible re-parameterization $L$'s vertex is typically
different from $O$.

\subsection{Pencil of quadrics in $\R^3$}
\label{quadric_pencil_r3}

The quartic curve $C$ from \eqref{quartic_on_s2:eq} on the sphere $F$
is the base locus of a pencil of quadrics $F+\l K$ in $\R^3$,
spanned by $F$ and the projection cone $K$ of $C$ from
its double point $\sv$, see Fig.~\ref{geom_prop:fig}.
The double point $\sv$ is located in the symmetry plane of $C$
and in the polar plane of the origin $O$ with respect to $F$. Its
coordinates are
\begin{equation}\label{doublepoint_quartic}
\sv=\frac{1}{m}(\g^2, 0, r\g) \mbox{ with } \g^2=m^2-r^2.
\end{equation}
The pencil $F+\l K$ contains two further singular quadrics which are obtained for the
zeros $\l_1=1/m$ and $\l_2=-1/\g$ of the characteristic polynomial
$$
\det(F+\l K) = r^2(m\l - 1)(\g \lambda + 1).
$$
Corresponding to $\l_1$ there is a parabolic cylinder $P$ with $y$-parallel generating lines passing through $C$. Corresponding to $\l_2$ we find the
rotational cone $L$ through $C$ with vertex $O$.

To give explicit representations for the quadrics we use homogeneous coordinates
$\yv=(1,x,y,z)^T$. Since there should not be any confusion, we use same notations
for the quadric $F$ and its coordinate
matrix appearing in the homogeneous quadratic equation $\yv^T\cdot F \cdot \yv=0$.
The coefficient matrices $F$ and $K$ read
\begin{equation}\label{sphere_cone:eq}
F=\left(
\begin{array}{cccc}
m^2-r^2 & -m & 0 & 0 \\
-m      & 1  & 0 & 0 \\
0       & 0  & 1 & 0 \\
0       & 0  & 0 & 1 \\
\end{array}
\right), \;
K=\left(
\begin{array}{cccc}
\gamma^{3} & -\gamma m & 0 & 0 \\
-\gamma m      & \gamma  & 0 & r \\
0       & 0  & -m & 0 \\
0       & r  & 0 & -\gamma \\
\end{array}
\right).
\end{equation}
An elementary computation shows that $K$ is a cone of revolution
with opening angle $\pi/2$ and $\av=(m+\g,0,r)$ denotes a direction
vector of its axis.

The cone $L$ through $C$ with vertex at $O$ is again a cone of revolution,
whose axis is parallel to $\av$. The parabolic cylinder $P$ through the
quartic $C$ has $y$-parallel generating lines. The axis of the cross section
parabola in the $xz$-plane is orthogonal
to $\av$, see Fig.~\ref{symcase_fig}.
The coefficient matrices $L$ and $P$ are
\begin{equation}\label{projection_cone:eq}
L = \left(
\begin{array}{cccc}
0 & 0 & 0 & 0 \\
0 & 0 & 0 & -r \\
0 & 0 & m+\g & 0 \\
0 & -r & 0 & 2\g\\
\end{array}
\right), \;
P =
\left(
\begin{array}{cccc}
\g^2(m+\g) & -m(m+\g) & 0 & 0 \\
-m(m+\g) & m+\g & 0 & r\\
0 & 0 & 0 & 0 \\
0 & r & 0 & m-\g \\
\end{array}
\right).
\end{equation}

A trigonometric parameterization of the quartic $C$ is obtained
by intersecting the cone $K$ with one quadric of the
pencil $F+\l K$, for instance $F$. Let $\av$ be a unit vector
in direction of $K$'s axis, and $\bv$ and $\cv$ complete it
to an orthonormal basis in $\R^3$.
A trigonometric parameterization of $K$ is given by
$$
\begin{array}{l}
\kv(t,v):=\sv + v(\av + (\bv \cos t + \cv\sin t )), \mbox{ with }\\
\av=\frac{1}{\sqrt{2m(m+\g)}}(m+\g,0,r), \;
\bv = (0,-1,0), \mbox{ and }
\cv = \frac{1}{\sqrt{2m(m+\g)}}(r,0,-(m+\g)).\\
\end{array}
$$
Thus $K$ admits the explicit parameterization
$$
\kv(t,v) =\frac{1}{2m\sqrt{m(m+\g)}}\left(
\begin{array}{c}
2\g^2\sqrt{m(m+\g)} + v\sqrt{2}m(m + \g + r\sin t)\\
-2 v m \sqrt{m(m+\g)} \cos t\\
2r\g\sqrt{m(m+\g)} + v\sqrt{2}m(r-(m+\g)\sin t)\\
\end{array}
\right).
$$
Finally, a trigonometric parameterization of the quartic $C$ follows by
\begin{equation}\label{trig_rep_quartic}
\cv(t) =\frac{1}{2m}\left(
\begin{array}{c}
(m+r\sin t)^2 + \g^2 \\
\sqrt{2}\sqrt{m(m+\g)}\cos t(\g - m  -r\sin t) \\
r(m+\g)\cos^2 t
\end{array}
\right), \mbox{ with }  \|\cv(t)\| = m+r\sin t.
\end{equation}
The correspondence of the trigonometric parameterization and its norm with
the expressions \eqref{quartic_on_s2:eq} and \eqref{ratdist1} in terms of rational
functions is realized by the substitutions $\sin t =(u^2-1)/(u^2+1)$ and
$\cos t = 2u/(u^2+1)$
and some rearrangement of the equations. Section~\ref{vivcrv:sec} discusses
relations to Viviani's curve (or Viviani's window).
This particular quartic has a similar shape and its pencil of quadrics
has similar properties. Viviani's curve has an additional symmetry.

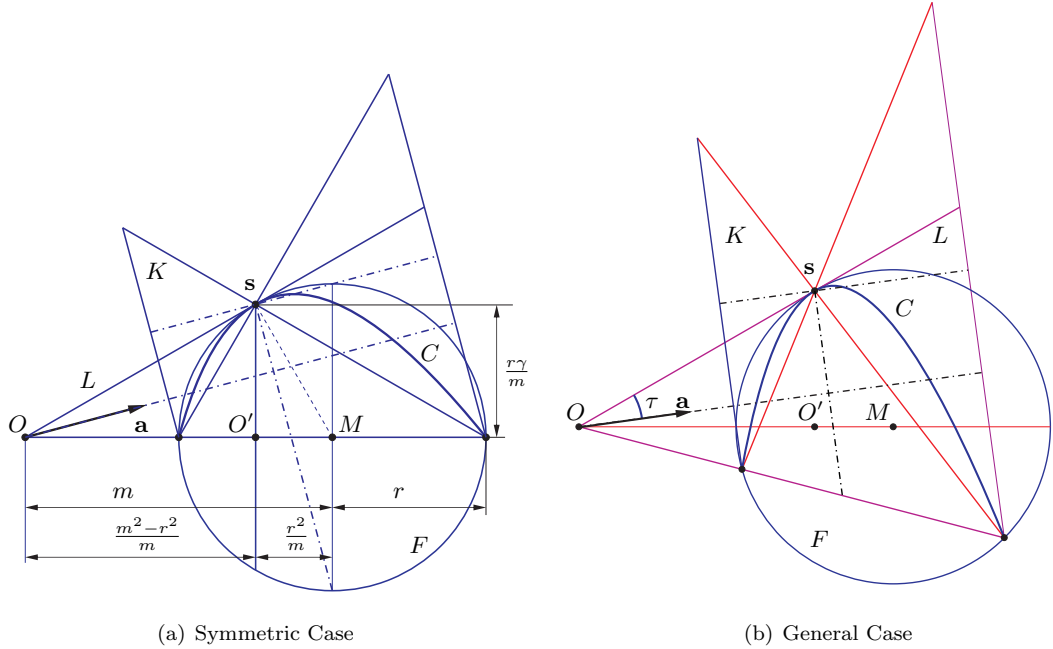
\begin{figure}[ht]
\subfigure[Symmetric Case]{\label{symcase_fig}
\begin{overpic}[width=.49\textwidth]{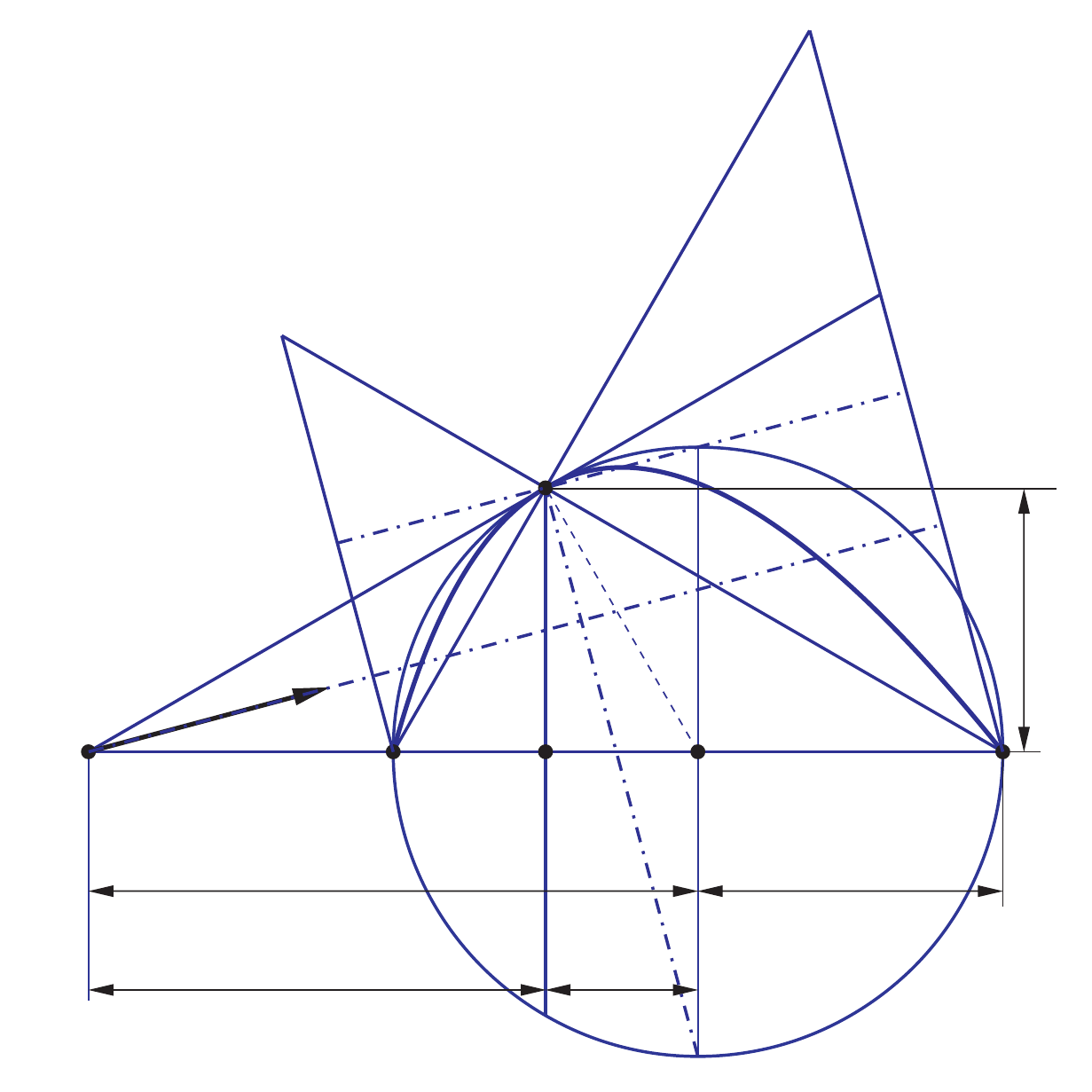}%
{\small
\put(78,10){$F$}
\put(80,45){$C$} \put(48,58){$\sv$} \put(28,32){$\av$}
\put(5,32){$O$} \put(45,32){$O'$} \put(65,32){$M$}
\put(18,40){$L$} \put(30,60){$K$}
\put(24,20){$m$} \put(24,12){$\frac{m^2-r^2}{m}$} \put(75,20){$r$}
\put(55,12){$\frac{r^2}{m}$}
\put(95,42){$\frac{r\g}{m}$}
}
\end{overpic}
}
\subfigure[General Case]{\label{gencase_fig}
\begin{overpic}[width=0.49\textwidth, angle=0]{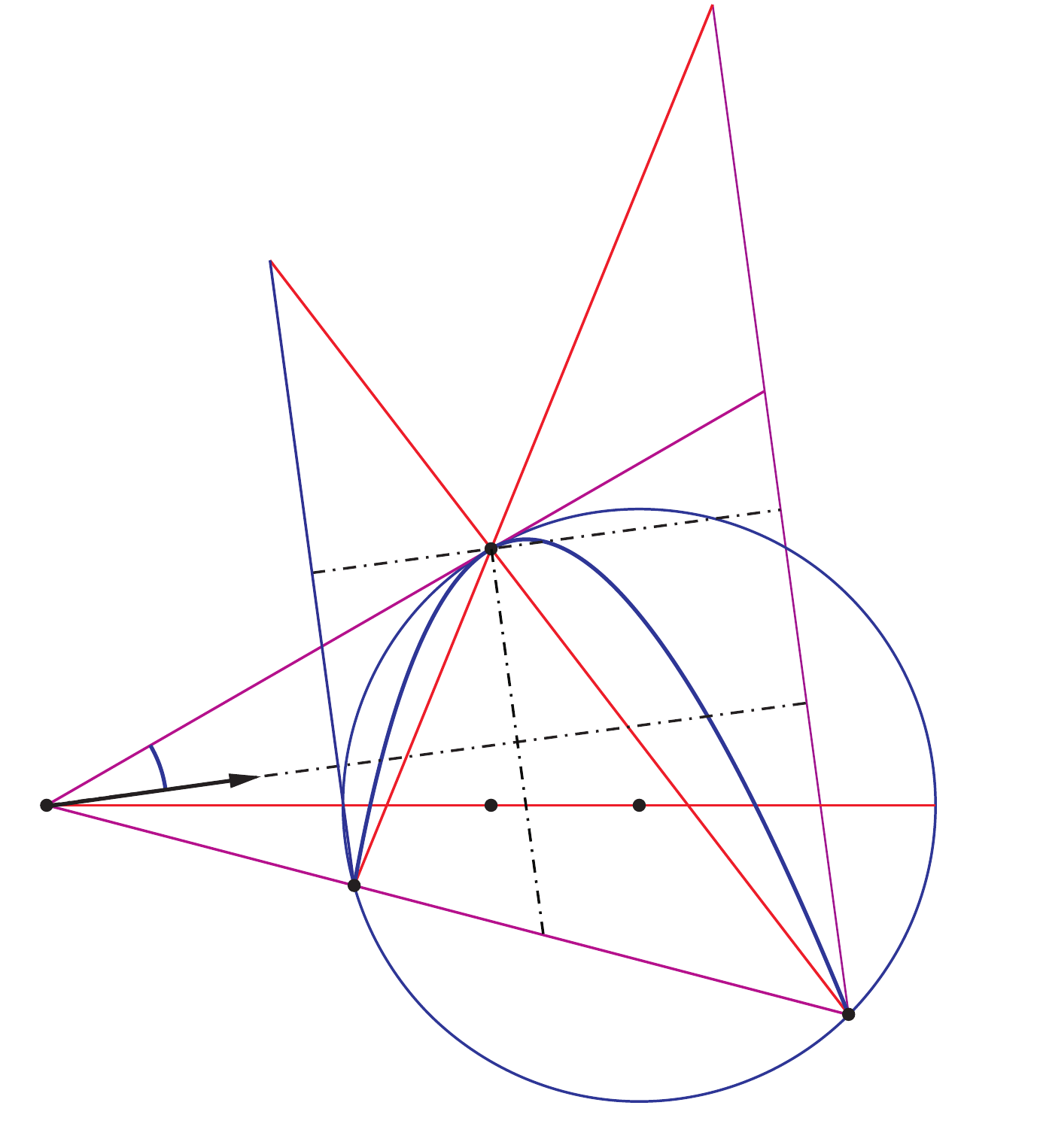}%
{\small
\put(42,10){$F$}
\put(56,48){$C$} \put(41,55){$\sv$} \put(20,33){$\av$}
\put(2,31){$O$} \put(40,31){$O'$} \put(51,31){$M$}
\put(62,60){$L$} \put(28,60){$K$}
\put(15,33){$\tau$}
}
\end{overpic}
}
\caption{Geometric properties of the conchoid construction}
\label{geom_prop:fig}
\end{figure}


\paragraph{Remark}

The inversion with center $O$ at the sphere which intersects the given sphere $F$
perpendicularly, maps the sphere $F$ onto itself. Analogously this inversion
fixes the rotational cone $L$.
Thus the quartic intersection curve $C=F\cap L$ remains fixed as a whole,
but of course not point-wise. The product of the distances
$\dist(O,P)$ and $\dist(O,P')$ of two
inverse points $P\in F$ and $P'\in F$ equals $\sqrt{m^2-r^2}$. This property
follows from the elementary tangent-secant-theorem of a circle.

\subsection{Relations to Viviani's curve}
\label{vivcrv:sec}

The quartic curve $C$, the base locus of the pencil of quadrics $F+tK$,
can be considered as generalization of Viviani's curve $V$.
This particularly well known curve $V$ is the base locus of a pencil of quadrics,
spanned by a sphere $F$ and a cylinder of revolution $L$ touching $F$
and passing through the center of $F$.
The pencil of quadrics of Viviani's curve also contains a right circular cone $K$ with
vertex in $V$'s double point and opening angle $\pi/2$, and further a parabolic
cylinder $P$. Viviani's curve $V$ is obtained from $C$ by letting $O\to\infty$.
Consequently, the inverse point $O'$ becomes the center of the sphere $F$.

Choosing the inverse point $O'=(\frac{m^2-r^2}{m},0,0)$ as origin, the
parameterization~\eqref{trig_rep_quartic} of $C$ becomes
\begin{equation}\label{quartic2:eq}
\cv(t) =\frac{1}{2m}\left(
\begin{array}{c}
r^2(1+\sin^2 t) + 2mr\sin t \\
\sqrt{2}\sqrt{m(m+\g)}\cos t(\g - m  -r\sin t) \\
r(m+\g)\cos^2 t
\end{array}
\right).
\end{equation}
By letting $m\to \infty$ one obtains $V$ as limit curve
\begin{equation}\label{quartic_viv:eq}
\vv(t) = \left(r\sin t, -r\sin t \cos t, r\cos^2 t \right).
\end{equation}

\begin{figure}
\subfigure[Quadric pencil of Viviani's curve]{\label{viv_pencil:fig}
\begin{overpic}[width=.45\textwidth]{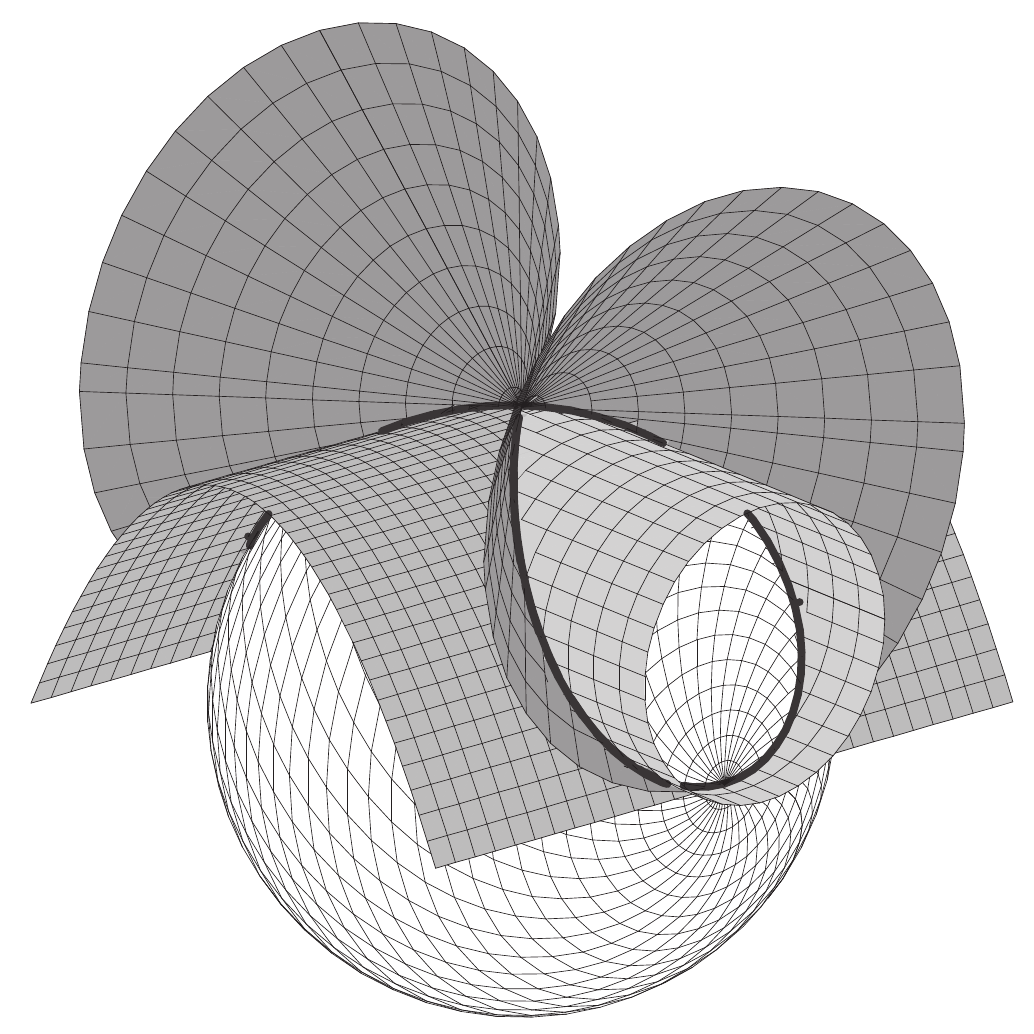}%
{\small
\put(55,80){$K$} \put(5,28){$P$} \put(21,10){$F$} \put(60,40){$L$}
\put(72,30){$V$}
}
\end{overpic}
}
\hfil
\subfigure[Quadric pencil of the quartic $C$]{\label{quadric_pencil:fig}
\begin{overpic}[width=.45\textwidth]{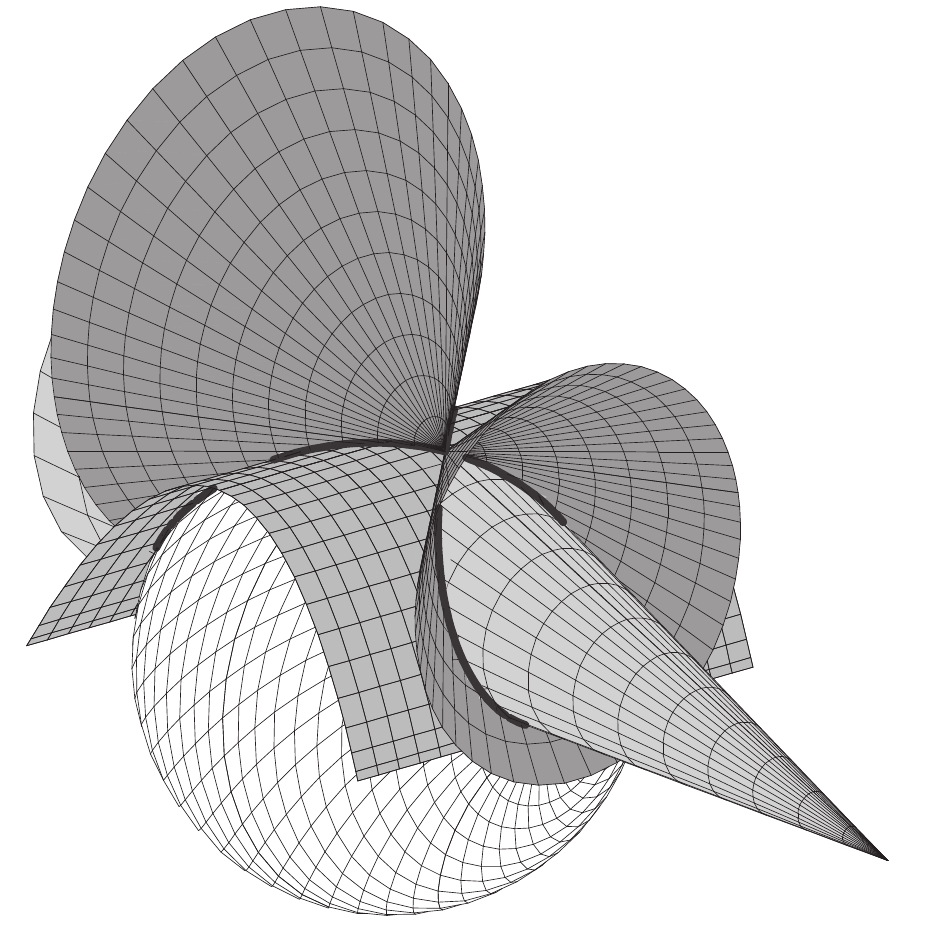}%
{\small
\put(53,80){$K$} \put(5,26){$P$} \put(20,5){$F$} \put(85,20){$L$}
}
\end{overpic}
}
\caption{Quadric pencils of Viviani's curve and its generalization}
\label{quad_pencil:fig}
\end{figure}
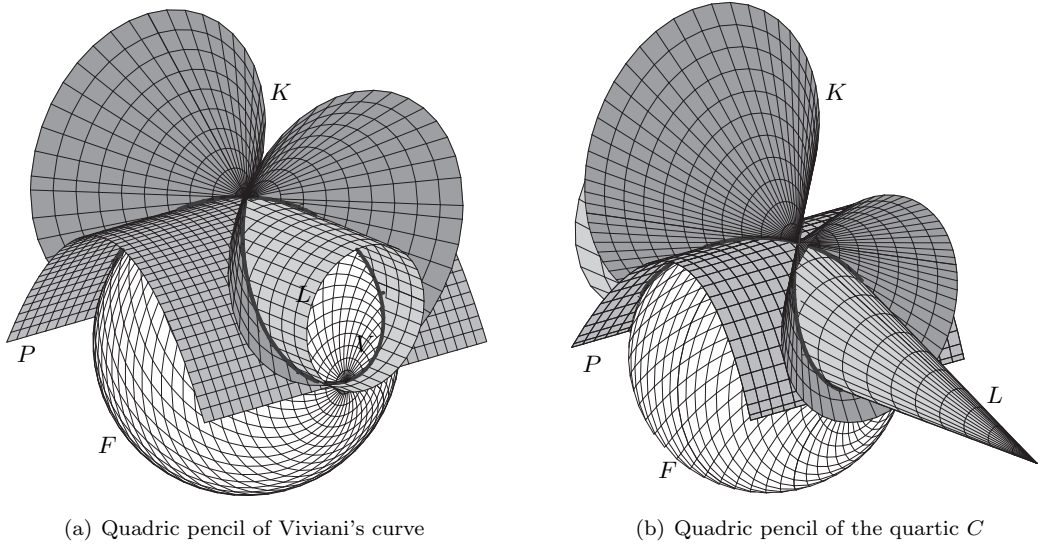

Fig.~\ref{viv_pencil:fig} illustrates Viviani's curve $V$,
together with the sphere and the singular quadrics belonging to the
pencil. The generalized Viviani curve $C$ being the base locus
of the pencil appearing in the conchoid construction of the sphere is illustrated
in Fig.~\ref{quadric_pencil:fig}.
In contrast to the classical Viviani curve $V$ whose single parameter $r$
is the radius of the sphere $F$, the quartic curve $C$ has two
parameters $r$ and $m$.


\section{Rotational quadrics with parallel axes}
\label{revcone:sec}

We consider the mentioned pencil of quadrics $Q(t)=A+tD$ from
Section~\ref{quad_pencil2:sec}, and a hyperplane
$E:ax+by+cz-dw=0$ passing through $O=(0,0,0,0)$.
The intersection $D\cap E$ is a quadratic cone whose projection onto $\R^3$
is a cone of revolution $L$ with axis in direction of $\av=(a,b,c)$. Assuming
$\|\av\|=1$, the opening angle $2\tau$ of $L$ is determined by $d=\cos\tau$.

Consider the quartic intersection curve $C=F\cap L$ of a sphere $F$ and the
cone of revolution $L$.
It is rational exactly if the cone $L$ is touching $F$ at a single point.
Since this touching point has to be contained in the polar plane of $O=(0,0,0)$
with respect to $F$, we choose $\sv=(\g^2/m, 0, r\g/m)$
(compare~\ref{doublepoint_quartic}) and prescribe
an arbitrary opening angle $2\tau$ for $L$. Thus the unit direction vector
of $L$'s axis is
$$
\av = \frac{1}{m}(\gamma\cos\tau - r\sin\tau, 0, \g\sin\tau + r\cos\tau)=(a,b,c).
$$
The quartic  $C$ is real if the axis is contained in
the wedge formed by $\sv$ and the $x$-axis, see Figure~\ref{gencase_fig}.
Thus $-r/\g\leq\tan\tau\leq0$,
because the rotation from $\sv$ to $\av$ by $\tau\leq 0$ is counterclockwise.
In the following we use the abbreviations $ct:=\cos\tau$ and $st:=\sin\tau$.
The quadrics of the pencil with base locus $C$ are denoted similarly to
Section~\ref{quadric_pencil_r3}. The coefficient matrix of the projection
cone $L$ reads
$$
L(\t) =
\left(
\begin{array}{cccc}
0 & 0 & 0 & 0 \\
0 & r^2(ct^2-st^2)+2\g r st ct  & 0 & -\g r (ct^2-st^2) + (r^2 - \g^2)st ct \\
0 & 0  & m^2ct^2 & 0 \\
0 & -\g r (ct^2-st^2) + (r^2 - \g^2)st ct  & 0 &
\g^2(ct^2 - st^2) -2\g r st ct  \\
\end{array}
\right).
$$
Rewriting $L(\t)$ in terms of the double angle $2\t$ and
substituting
\begin{equation}\label{double_angle_subs}
\cos2\tau=\g/m, \mbox{ and }  \sin2\tau=-r/m
\end{equation}
we obtain $L$ from equation~\eqref{projection_cone:eq}.
This holds for all equations and parameterizations in this section
in an analogous way.

The pencil of quadrics $F+tL(\tau)$ contains two further singular quadrics. The first is a parabolic
cylinder $P(\tau)$ passing through $C$.
It corresponds to the eigenvalue $\frac{-1}{m^2 ct^2}$
and its generating lines are parallel to the $y$-axis.
Its coefficient matrix of cylinder reads
$$
P(\t) =
\left(
\begin{array}{cccc}
\g^2m^2 ct^2 & -m^3 ct^2 & 0 & 0 \\
-m^3 ct^2 & \g^2(ct^2-st^2) + m^2 st^2 - 2 r\g st ct  & 0 &
(\g^2 - r^2) st ct + r\g (ct^2-st^2) \\
0 & 0  & 0 & 0 \\
0 & (\g^2 - r^2) st ct + r\g (ct^2-st^2)  & 0 &
m^2 ct^2 -\g^2(ct^2 - st^2) + 2r\g st ct  \\
\end{array}
\right).
$$

Our goal is not only to characterize the pencil of quadrics but to provide
an explicit parameterization of the quartic curve $C$ on $F$ whose distance
from $O$ is rational. This is performed by using a parameterization of the second
singular quadric $K$ which corresponds to the zero
$\frac{r}{\g m^2 ct st}$ of the characteristic polynomial $\det(F+tL(\tau))$.
$K$ is a cone of revolution with axis parallel to $\av$, and its coefficient matrix
reads
$$
K(\t) =
\left(
\begin{array}{cccc}
\g^2 & -m & 0 & 0 \\
-m & \frac{\g (m^2+2r^2) st ct + r^3 (ct^2 - st^2)}{\g m^2 st ct}  & 0 &
\frac{-r((\g^2 - r^2) st ct + \g r (ct^2 - st^2))}{\g m^2 st ct} \\
0 & 0  & \frac{\g st+r ct}{\g st} & 0 \\
0 & \frac{-r((\g^2 - r^2) st ct + \g r (ct^2 - st^2))}{\g m^2 st ct}  & 0 &
\frac{\g (\g^2 - r^2) st ct + r \g^2 (ct^2 - st^2)}{\g m^2 st ct}  \\
\end{array}
\right).
$$

A parameterization of the cone of revolution $K$ with respect to
its vertex $\sv$ is
$$
\kv(u,v) = \sv + v(\av + R(\bv\cos u + \cv\sin u)),
$$
where $\av$ is a unit vector in direction of its axis, and $\bv$ and $\cv$
complete $\av$ to an orthonormal basis in $\R^3$, and $R$ denotes the radius
of the cross section circle at distance 1 from $\sv$ which has still to be determined.
In detail this reads
$$
\kv(u,v) = \left(
\begin{array}{ccc}
\frac{\g^2}{m} + v(\frac{\g ct - r st}{m} + R\frac{\sin u(\g st + r ct)}{m}) \\
-vR\cos u\\
\frac{\g r}{m} + v(\frac{\g st + r ct}{m} + R\frac{\sin u(-\g ct + r st)}{m}) \\
\end{array}
\right).
$$
Inserting $\kv(u,v)$ into the equation $\yv^T\cdot K(\t) \cdot\yv=0$ defines the radius
$$
R = \frac{\sqrt{-ct st(\g st + r ct)(\g ct - r st)}}{ct(\g st + r ct)} =
\sqrt{\frac{- st(\g ct - r st)}{ct(\g st + r ct)}}.
$$
The final parameterization of the quartic curve $C$ is obtained for
$v=\frac{2r(R\sin u ct - st)}{1+R^2}$ and is a bit lengthy.
It reads
\begin{equation}\label{ratdistquartic}
\cv(u) = \left(
\begin{array}{ccc}
\frac{(4 R r\sin u ct(\g ct - r st) + 2r ct (\g st + r ct)(R^2\sin^2u-1)
+ m^2 + r^2 + R^2(m^2 - r^2) - 2 Rr\g\sin u )}{m(1+R^2)}\\
\frac{-2Rr\cos u(R ct \sin u - st)}{1+R^2}\\
\frac{r( 2R^2 ct \sin^2 u ( r st - \g ct) +  4 R\g\sin u  ct st  - \g(1 - R^2)
 - 2r ct st + 2\g ct^2 + 2Rr\sin u (ct^2 - st^2))}{m(1+R^2)}
\end{array}
\right),
\end{equation}
and its norm is
$$
\|\cv(u)\| = \frac{\g ct(1+R^2) - 2r st + 2r R ct \sin(u)}{ct (1+R^2)}.
$$
This is proved by using the incidence $\cv\subset E$, thus
$a\cv_1+b\cv_2+c\cv_3=ct w$, with $w=\|\cv\|$.
Note that $R$ is {\it not rational} in any rational substitution for the
trigonometric functions $\cos\tau$ and $\sin\tau$.
Rotating $C$ around the $x$-axis gives a rational polar representation
$\fv(u,v)$ of the sphere $F$.
The resulting parameterization $\fv$ of $F$ is not proper, but
almost all points of $F$ are traced twice, therefore belonging to two parameter
values $(u_1,v)$ and $(u_2, v)$. We summarize the construction.

\begin{corollary}
There exists a one-parameter family of quartic curves $C(\tau)\subset F$
with double point at $\sv$ and symmetry plane $y=0$. The corresponding
pencils of quadrics $Q(t)=F+\lambda L(\tau)$ contain rotational cones
$K(\tau)$ and $L(\tau)$, where the vertex of the latter is at $O$, and
a parabolic cylinder $P(\tau)$. Besides $P(\tau)$ all quadrics have rotational
symmetry with parallel axes $\av(\tau)$. The distance function $\dist(OC)=\|\cv(u)\|$
is rational in the curve parameter, but not rational in the angle-parameter
$\tau$.
\end{corollary}

\section{Conclusion}

We have discussed the conchoid construction for spheres and have shown
that a sphere in $\R^3$ admits a rational polar representation with respect
to an arbitrary chosen focus point, which implies that the conchoid
surfaces of spheres possess rational parameterizations.
Additionally we have given a geometric
construction for these parameterizations which are based on a rational curve
of degree four being the base locus of a pencil of quadrics in $\R^3$.
Relations to the classical Viviani curve have been addressed.
The construction of the rational parameterization of the conchoids
is also based on a pencil of quadrics in $\R^4$.

\section*{Acknowledgments}
This work was developed, and partially supported, under the research
project MTM2008-04699-C03-01 "Variedades paramétricas: algoritmos y
aplicaciones", Ministerio de Ciencia e Innovación, Spain and by "
Fondos Europeos de Desarrollo Regional" of the European Union.


\end{document}